\documentclass[10pt,reqno]{amsart}
\usepackage{latexsym,amsmath,amssymb,amscd}
\usepackage[all]{xy}

\def\today{\ifcase \month \or
   January \or February \or March \or April \or
   May \or June \or July \or August \or
   September \or October \or November \or December \fi
   \space\number\day , \number\year}
   
\makeatletter
  \newcommand\@dotsep{4.5}
  \def\@tocline#1#2#3#4#5#6#7{\relax
     \ifnum #1>\c@tocdepth % then omit
     \else
     \par \addpenalty\@secpenalty\addvspace{#2}%
     \begingroup \hyphenpenalty\@M
     \@ifempty{#4}{%
     \@tempdima\csname r@tocindent\number#1\endcsname\relax
        }{%
         \@tempdima#4\relax
           }%
      \parindent\z@ \leftskip#3\relax \advance\leftskip\@tempdima\relax
      \rightskip\@pnumwidth plus1em \parfillskip-\@pnumwidth
       #5\leavevmode\hskip-\@tempdima #6\relax
       \leaders\hbox{$\m@th
       \mkern \@dotsep mu\hbox{.}\mkern \@dotsep mu$}\hfill
       \hbox to\@pnumwidth{\@tocpagenum{#7}}\par
       \nobreak
        \endgroup
         \fi}
\makeatother 
\begin{document}

%\DeclareRobustCommand{\SkipTocEntry}[4]{} 

\makeatletter
\@addtoreset{figure}{section}
\def\thefigure{\thesection.\@arabic\c@figure}
\def\fps@figure{h,t}
\@addtoreset{table}{bsection}

\def\thetable{\thesection.\@arabic\c@table}
\def\fps@table{h, t}
\@addtoreset{equation}{section}
\def\theequation{%\thesection.
\arabic{equation}}
\makeatother

\newcommand{\bfi}{\bfseries\itshape}

\newtheorem{theorem}{Theorem}
\newtheorem{acknowledgment}[theorem]{Acknowledgment}
\newtheorem{corollary}[theorem]{Corollary}
\newtheorem{definition}[theorem]{Definition}
\newtheorem{example}[theorem]{Example}
\newtheorem{lemma}[theorem]{Lemma}
\newtheorem{notation}[theorem]{Notation}
\newtheorem{proposition}[theorem]{Proposition}
\newtheorem{remark}[theorem]{Remark}
\newtheorem{setting}[theorem]{Setting}

\numberwithin{theorem}{section}
\numberwithin{equation}{section}

\renewcommand{\1}{{\bf 1}}
\newcommand{\Ad}{{\rm Ad}}
\newcommand{\alg}{{\rm alg}}
\newcommand{\ad}{{\rm ad}}
\newcommand{\botimes}{\bar\otimes}
\newcommand{\Ci}{{\mathcal C}^\infty}
\newcommand{\Cint}{{\mathcal C}^\infty_{\rm int}}
\newcommand{\de}{{\rm d}}
\newcommand{\ee}{{\rm e}}
\newcommand{\End}{{\rm End}\,}
\newcommand{\id}{{\rm id}}
\newcommand{\ie}{{\rm i}}
\newcommand{\ind}{\mathop{\rm ind}}
\newcommand{\Hom}{{\rm Hom}\,}
\newcommand{\Img}{{\rm Im}\,}
\newcommand{\Ind}{{\rm Ind}}
\newcommand{\Ker}{{\rm Ker}\,}
\newcommand{\Lie}{\text{\bf L}}
\newcommand{\Mt}{{{\mathcal M}_{\text t}}}
\newcommand{\Ran}{{\rm Ran}\,}
\renewcommand{\Re}{{\rm Re}\,}
\newcommand{\spa}{{\rm span}\,}
\newcommand{\Tr}{{\rm Tr}\,}
\newcommand{\tw}{\ast_{\rm tw}}
\newcommand{\Op}{{\rm Op}}
\newcommand{\U}{{\rm U}}

\newcommand{\UC}{{{\mathcal U}{\mathcal C}}}
\newcommand{\UCb}{{{\mathcal U}{\mathcal C}_b}}
\newcommand{\LUCb}{{{\mathcal L}{\mathcal U}{\mathcal C}_b}}
\newcommand{\RUCb}{{{\mathcal R}{\mathcal U}{\mathcal C}_b}}

\newcommand{\CC}{{\mathbb C}}
\newcommand{\RR}{{\mathbb R}}
\newcommand{\TT}{{\mathbb T}}

\newcommand{\Ac}{{\mathcal A}}
\newcommand{\Bc}{{\mathcal B}}
\newcommand{\Cc}{{\mathcal C}}
\newcommand{\Dc}{{\mathcal D}}
\newcommand{\Fc}{{\mathcal F}}
\newcommand{\Hc}{{\mathcal H}}
\newcommand{\Lc}{{\mathcal L}}
\renewcommand{\Mc}{{\mathcal M}}
\newcommand{\Nc}{{\mathcal N}}
\newcommand{\Qc}{{\mathcal Q}}
\newcommand{\Vc}{{\mathcal V}}
\newcommand{\Wig}{{\mathcal W}}
\newcommand{\Xc}{{\mathcal X}}
\newcommand{\Yc}{{\mathcal Y}}

\newcommand{\Lg}{{\mathfrak L}}
\newcommand{\Sg}{{\mathfrak S}}

\renewcommand{\gg}{{\mathfrak g}}
\newcommand{\hg}{{\mathfrak h}}

\newcommand{\ZZ}{\mathbb Z}
\newcommand{\NN}{\mathbb N}
\newcommand{\BB}{\mathbb B}
\newcommand{\HH}{\mathbb H}

\newcommand{\hake}[1]{\langle #1 \rangle }

\newcommand{\scalar}[2]{\langle #1 ,#2 \rangle }
\newcommand{\vect}[2]{(#1_1 ,\ldots ,#1_{#2})}
\newcommand{\norm}[1]{\Vert #1 \Vert }
\newcommand{\normrum}[2]{{\norm {#1}}_{#2}}

\newcommand{\upp}[1]{^{(#1)}}
\newcommand{\p}{\partial}

\newcommand{\opn}{\operatorname}
\newcommand{\slim}{\operatornamewithlimits{s-lim\,}}
\newcommand{\sgn}{\operatorname{sgn}}

\newcommand{\seq}[2]{#1_1 ,\dots ,#1_{#2} }
\newcommand{\loc}{_{\opn{loc}}}

\makeatletter
\title[On Wigner transforms in infinite dimensions]{On Wigner transforms in infinite dimensions}
\author{Ingrid Belti\c t\u a,   
Daniel Belti\c t\u a, 
and Marius M\u antoiu}
\address{Institute of Mathematics ``Simion Stoilow'' 
of the Romanian Academy, 
P.O. Box 1-764, Bucharest, Romania}
\email{ingrid.beltita@gmail.com, Ingrid.Beltita@imar.ro}
\address{Institute of Mathematics ``Simion Stoilow'' 
of the Romanian Academy, 
P.O. Box 1-764, Bucharest, Romania}
\email{beltita@gmail.com, Daniel.Beltita@imar.ro}
\address{Departamento de Matem\'aticas, Universidad de Chile, Las Palmeras 3425,
Casilla 653, Santiago, Chile}
\email{mantoiu@uchile.cl}
%\translator{}
%\dedicatory{BBM2\_Jan21\_2015.tex}
\thanks{The research of I. Belti\c t\u a and D. Belti\c t\u a has been partially supported by the Grant
of the Romanian National Authority for Scientific Research, CNCS-UEFISCDI,
project number PN-II-ID-PCE-2011-3-0131. 
M. M\u antoiu has been supported by the Fondecyt Project 1120300 and the N\'ucleo Milenio de F\'isica Matem\'atica RC120002.}
\date{January 22, 2015}%{\today}
\makeatother

\begin{abstract}
We investigate the Schr\"odinger representations of certain infinite-dimensional Heisenberg groups, 
using their corresponding Wigner transforms. 
\end{abstract}

\maketitle

\tableofcontents

\section{Introduction}

The topic of this paper belongs to representation theory of Heisenberg groups, 
more specifically we investigate to what extent the square-integrability properties of the Schr\"odinger representations 
carry over to the setting of infinite-dimensional Heisenberg groups. 
Recall that 
the Schr\"odinger representation of the $(2n+1)$-dimensional Heisenberg group $\HH_{2n+1}$ 
is a group representation 
$$\pi\colon \HH_{2n+1}\to\Bc(L^2(\RR^n,\de\lambda))$$ 
and its corresponding Wigner transform is a unitary operator 
$$\Wig\colon L^2(\RR^n,\de\lambda)\bar\otimes\overline{L^2(\RR^n,\de\lambda)}\to 
L^2(\RR^n\times\RR^n,\de\mu)$$
where we denote by $\de\lambda$ the Lebesgue measure on $\RR^n$, by $\Bc(L^2(\RR^n,\de\lambda))$ 
the bounded linear operators on the complex Hilbert space $L^2(\RR^n,\de\lambda)$, 
and by $\de\mu$ a suitable normalization of the measure $\de\lambda\otimes\de\lambda$ 
on $\RR^n\times\RR^n$. 
One way to define the Wigner transform is that of defining $\Wig(f\otimes\overline{\varphi})\in L^2(\RR^n\times\RR^n,\de\mu)$ 
as a Fourier transform of the representation coefficient 
$(\pi(\cdot)f\mid\varphi)\vert_{\RR^n\times\RR^n\times\{0\}}\in L^2(\RR^n\times\RR^n\times\{0\},\de\lambda\otimes\de\lambda)$, 
recalling that $\HH_{2n+1}=\RR^n\times\RR^n\times\RR$ as smooth manifolds. 

As the translation invariance property of the Lebesgue measure plays a central role in the above discussion, 
it is not straightforward to replace here $\RR^n$ by an infinite-dimensional real Hilbert space.  
It is customary in the infinite-dimensional analysis to try to replace the Lebesgue measure by a Gaussian measure, 
and this is what we will do in the present paper as well. 
Furthermore, the main problems that we must address are to construct the Wigner transform 
as a unitary operator on square-integrable functions of infinitely many variables and  
to realize the image of that unitary operator as an $L^2$-space, 
which amounts to determining the infinite-dimensional analogue 
of the measure~$\de\mu$ from the above paragraph. 
In some sense, these problems form a complement to the ones addressed in our recent investigation 
of square-integrable families of operators~\cite{BBM14}. 
 
The present paper is organized as follows. 
Sections \ref{Sect2} and \ref{Sect3} develop 
an abstract framework for the study of Wigner transforms associated to unitary representations 
of general topological groups. 
The main ingredients of that framework are the Fourier transforms on uniform spaces and 
an operator calculus that involves the Banach algebra structures 
of the dual of the spaces of left uniformly continuous functions on topological groups. 
Then Section~\ref{Sect4} records some computations with Gaussian functions 
and their Wigner transforms. 
In Section~\ref{Sect5} we introduce the infinite-dimensional Heisenberg groups and their Schr\"odinger representations 
for which we construct their corresponding Wigner transforms in Theorem~\ref{orth_th}. 

Throughout this paper we denote by $\bar\otimes$ the Hilbertian scalar product and 
by $\Bc(\Xc)$ and $\Xc'$ the spaces of all bounded linear operators and bounded linear functionals 
on some Banach space~$\Xc$, respectively, 
and it will always be clear from the context if the ground field is $\RR$ or $\CC$.

\section{Preliminaries on Fourier transforms on uniform spaces}\label{Sect2}

Let $X$ be any Hausdorff uniform space. 
We denote by  $\UC$ and $\UCb$ the various spaces of uniformly continuous functions 
and uniformly continuous bounded functions on any uniform space, respectively. 
Note that $\UCb(X):=\UCb(X,\CC)$ is a Banach space, 
so we may consider its dual Banach space $\Mc(X):=\UCb(X)'$,  
which should be thought of as a space of generalized complex measures on $X$.

\begin{definition}\label{gen1}
\normalfont
We define the \emph{uniform space dual}\index{dual!uniform space} 
to $X$ as 
$$X^\nabla:=\UC(X,\RR)$$ 
endowed with the uniform structure of pointwise convergence. 
\end{definition}

\begin{remark}\label{gen2}
\normalfont
There exists a natural injective mapping 
$$\eta_X\colon X\to(X^\nabla)^\nabla,\quad x\mapsto\eta_x$$
where $\eta_x(f)=f(x)$ for every $f\in\UC(X,\RR)$ and $x\in X$. 
It is easily seen that $\eta$ is uniformly continuous. 
\end{remark}

\begin{definition}\label{gen3}
\normalfont
The \emph{Fourier transform}\index{Fourier transform!on uniform spaces} 
on $X$ is the linear mapping  
$$\Fc\colon\Mc(X)\to\ell^\infty(X^\nabla),\quad 
(\Fc\mu)(f)=\langle\mu,\ee^{\ie f}\rangle$$ 
for $f\in X^\nabla$ and $\mu\in\Mc(X)$. 
We will also denote $\widehat{\mu}:=\Fc\mu$ for $\mu\in\Mc(X)$. 

Note that $\Fc$ is a bounded linear mapping and in fact $\Vert\Fc\Vert\le1$. 
\end{definition}

\begin{proposition}\label{gen3.5}
The Fourier transform $\Fc\colon\Mc(X)\to\ell^\infty(X^\nabla)$ is injective. 
\end{proposition}

\begin{proof}
Let $\mu\in\Mc(X)$ with $\Fc\mu=0$. 
In order to prove that $\mu=0$, it suffices to show that for an arbitrary 
\emph{real-valued} function $f\in\UCb(X)$ we have $\langle\mu,f\rangle=0$. 
First recall that we have 
$$\lim\limits_{t\to 0}\frac{\ee^{\ie tr}-1}{t}-\ie r=0 $$
uniformly for $r$ in any compact subset of~$\RR$. 
Therefore 
$$\lim\limits_{t\to 0}\frac{\ee^{\ie tf}-1}{t}=\ie f $$
in $\UCb(X)$, hence 
$$\ie\langle\mu,f\rangle=\lim\limits_{t\to 0}\frac{(\Fc\mu)(tf)-(\Fc\mu)(0)}{t}=0,$$
and we are done. 
\end{proof}

\begin{lemma}\label{gen4}
If $\mu\in\Mc^+(X)$ then the following assertions hold: 
\begin{enumerate}
\item\label{gen4_item1} 
We have $\Vert\Fc\mu\Vert_\infty\le\langle\mu,\1\rangle$. 
\item\label{gen4_item2} 
For all $f,h\in X^\nabla$ we have 
$$\vert(\Fc\mu)(f)-(\Fc\mu)(h)\vert^2\le 
2\langle\mu,\1\rangle(\langle\mu,\1\rangle-\Re(\Fc\mu)(f-h)). $$
\end{enumerate}
\end{lemma}

\begin{proof}
For Assertion~\eqref{gen4_item1} recall from Definition~\ref{gen3} that $\Vert\Fc\Vert\le1$, 
hence $$\Vert\Fc\mu\Vert_\infty\le\Vert\mu\Vert=\langle\mu,\1\rangle$$
where the latter equality follows since $\mu\colon\UCb(X)\to\CC$ is a positive linear functional on 
the $C^*$-algebra $\UCb(X)$.

To prove Assertion~\eqref{gen4_item2}, let $f,h\in X^\nabla$ arbitrary. 
By using the Cauchy-Schwartz inequality  we get 
\allowdisplaybreaks
\begin{align}
\vert(\Fc\mu)(f)-(\Fc\mu)(h)\vert^2 
&= \vert\langle\mu,\ee^{\ie f}-\ee^{\ie h}\rangle\vert^2 \nonumber \\
&\le \langle\mu,\vert\ee^{\ie f}-\ee^{\ie h}\vert^2\rangle 
\langle\mu,\1^2\rangle \nonumber \\
&=\langle\mu,\1\rangle \langle\mu,2-2\Re(\ee^{\ie(f-h)})\rangle \nonumber \\
&=2\langle\mu,\1\rangle(\langle\mu,\1\rangle-\langle\mu,\Re(\ee^{\ie(f-h)})\rangle)  \nonumber \\
&=2\langle\mu,\1\rangle(\langle\mu,\1\rangle-\Re((\Fc\mu)(f-h))),\nonumber
\end{align}
where we also used the fact that 
$\vert \ee^{\ie t}-\ee^{\ie s}\vert^2=2-2\Re(\ee^{\ie(t-s)})$ for all $t,s\in\RR$. 
\end{proof}

\begin{lemma}\label{map5}
Let $X$ and $Y$ be uniform spaces. 
Assume that $D$ is a dense subset of $X$ and 
$\Ac$ 
is a uniformly equi-continuous family of mappings from $X$ into $Y$.
Then the following uniform structures on the mappings from $X$ into $Y$  
induce the same uniform structure on $\Ac$:
\begin{itemize}
\item[{\rm(a)}] the structure of uniform convergence on 
the precompact subsets of $X$; 
\item[{\rm(b)}] the structure of pointwise convergence on $X$; 
\item[{\rm(c)}] the structure of pointwise convergence on $D$. 
\end{itemize}
\end{lemma}

\begin{proof}
See the proof of the Ascoli-Arzel\`a theorem on uniform spaces in \cite{Bo69}.  
\end{proof}

We now introduce the linear space of tight measures $\Mt(X)$ on a uniform space~$X$  
(see \cite[Sect. 5.1]{Pa13} for more information in this connection). 
Namely, $\Mt(X)$ is the set of all linear functionals $\varphi\colon\UCb(X)\to{\mathbb C}$  
with the property that for every net 
$\{f_i\}_{i\in I}$ in $\UCb(X)$ with 
$\sup\limits_{i\in I}\Vert f_i\Vert_\infty<\infty$ and 
$\lim\limits_{i\in I}f_i=0$ 
uniformly on every compact subset of~$X$ 
one has $\lim\limits_{i\in I}\varphi(f_i)=0$. 
We also denote 
$\Mt^{+}(X):=\Mt(X)\cap\Mc^{+}(X)$. 

\begin{lemma}\label{class6.5}
Let $X$ be a uniform space and $\varphi\colon\UCb(X)\to{\mathbb C}$ 
a self-adjoint linear functional with its positive part and negative part~$\varphi^{\pm}$. 
Then $\varphi\in \Mt(X)$  
if and only if $\varphi^{\pm}\Mt(X)$. 
\end{lemma}

\begin{proof}
See \cite{Feo67}. 
\end{proof}

\begin{proposition}\label{gen5} 
The following assertions hold: 
\begin{enumerate}
\item\label{gen5_item1} 
Let $\mu\in\Mc^+(X)$. 
We have $\Fc\mu\in\UCb(X^\nabla)$ if and only if $\Re(\Fc\mu)$ is continuous at $0\in X^\nabla$. 
\item\label{gen5_item2} 
If $\mu\in\Mt(X)$, then $\Fc\mu\in\UCb(X^\nabla)$. 
\end{enumerate}
\end{proposition}

\begin{proof}
Assertion~\eqref{gen5_item1} follows at once by Lemma~\ref{gen4}. 

For proving Assertion~\eqref{gen5_item2}, we see from Lemma~\ref{class6.5} that we may assume $\mu\in\Mt^+(X)$.  
Then, according to Assertion~\eqref{gen5_item1}, it suffices to show that 
for every $\mu\in\Mt^+(X)$ the function $\Re(\Fc\mu)\colon X^\nabla\to\CC$ is continuous at $0\in X^\nabla$. 
To this end, let us assume that $\lim\limits_{j\in J}f_j=0$ in $X^\nabla$. 
In other words, $\{f_j\}_{j\in J}$ is a net of uniformly continuous real functions on $X$ 
with $\lim\limits_{j\in J}f_j=0$ pointwise on $X$. 
By using Lemma~\ref{map5}, we see that $\{\cos f_j\}_{j\in J}$ is 
a uniformly bounded net in $\UCb(X)$ which converges to $1\in \UCb(X)$ uniformly on the compact subsets of $X$. 
Since $\mu\in\Mt(X)$, we then get 
$$\lim\limits_{j\in J}\Re(\Fc\mu)(f_j)
=\lim\limits_{j\in J}\Re\langle\mu,\ee^{\ie f_j}\rangle
=\lim\limits_{j\in J}\langle\mu,\cos f_j\rangle
=1. $$
Therefore the function $\Re(\Fc\mu)$ is continuous at $0\in X^\nabla$, and this completes the proof. 
\end{proof}

\begin{remark}
\normalfont
Lemma~\ref{gen4} and Proposition~\ref{gen5} are straightforward extensions of 
some results from~\cite[\S 6, no.~8]{Bo69}. 
\end{remark}

\section{Operator calculus on topological groups}\label{Sect3}

In this section we introduce an operator calculus for unitary representations of topological groups, 
since it will allow us to handle in Section~\ref{Sect5} some representations of infinite-dimensional Heisenberg groups,  
which are Lie groups modeled on Hilbert spaces. 
In the case of Banach-Lie groups,  
the space of continuous 1-parameter subgroups and exponential map, 
as defined below,  
agree with the usual notions of Lie algebra and exponential map from the Lie theoretic setting 
(see \cite{Nee06} for extensive information in this connection). 
We also note that in the case of finite-dimensional nilpotent Lie groups 
the present operator calculus recovers the Weyl calculus used in \cite{GH14} in the study of L\'evy processes.

Let $G$ be a topological group endowed with the right uniform structure. 
We recall that a basis of this uniform structure is provided by the sets 
$$\quad S^\lambda_V
=\{(x,y)\in G\times G\mid yx^{-1}\in V\},$$ 
where $V\in\Vc_G(\1)$. 
Consider the corresponding space of mappings 
$$\Lg(G)=\{X\colon{\mathbb R}\to G\mid X\text{ homomorphism of topological groups}\}$$ 
with the structure of uniform convergence on the compact subsets of $\RR$. 
Hence a basis of this uniform structure consists of the entourages 
$$\quad S^\lambda_{n,V}=\{(X,Y)\in\Lg(G)\times\Lg(G)\mid
(\forall t\in[-n,n])\quad 
Y(t)X(t)^{-1}\in V\} $$
parameterized by $V\in\Vc_G(\1)$ and $n\in\NN$ 
(see \cite[Def. 2.6]{HM07}). 

It is easily seen that the exponential mapping 
$$\exp_G\colon\Lg(G)\to G,\quad \exp_G X:=X(1)$$
is uniformly continuous, hence it gives rise to 
a unital $*$-homomorphism of $C^*$-algebras 
$$\UCb(\exp_G)\colon\LUCb(G)\to\UCb(\Lg(G)),\quad f\mapsto f\circ\exp_G$$
with the dual map 
$$\UCb'(\exp_G)\colon(\UCb(\Lg(G)))'\to(\LUCb(G))'. $$
We recall from \cite{Gr} 
(see also \cite[Th. 3.9]{B11}) 
that $(\LUCb(G))'$ is a unital associative Banach algebra in a natural way and every continuous unitary representation 
$\pi\colon G\to\Bc(\Hc_\pi)$ gives rise to a Banach algebra representation 
$$\widehat{\pi}_{\Lc}\colon(\LUCb(G))'\to\Bc(\Hc)$$
such that 
$(\widehat{\pi}_{\Lc}(\nu)\phi\mid\psi)
=\langle\nu,(\pi(\cdot)\phi\mid\psi)\rangle$ 
for $\phi,\psi\in\Hc$ and $\nu\in(\LUCb(G))'$. 
Therefore we get a bounded linear operator 
$$\widetilde{\pi}_{\Lc}\colon(\UCb(\Lg(G)))'\to\Bc(\Hc)$$
such that the diagram 
$$
\xymatrix{
(\LUCb(G))' \ar[r]^{\widehat{\pi}_{\Lc}} & \Bc(\Hc) \\
(\UCb(\Lg(G)))' \ar[u]^{\UCb'(\exp_G)} \ar[ur]_{\widetilde{\pi}_{\Lc}}& 
}
$$
is commutative. 

\begin{setting}\label{opcalc0}
\normalfont
Throughout the rest of this section we fix a continuous unitary representation $\pi\colon G\to\Bc(\Hc)$ 
of the above topological group $G$ 
on the complex Hilbert space $\Hc$ 
and we assume the setting defined by the following data: 
\begin{itemize} 
\item a uniform space $\Xi$ 
and a uniformly continuous map $\theta\colon\Xi\to\Lg(G)$, 
\item a locally convex space $\Gamma$ such that 
there exists an injective continuous inclusion map 
$\Gamma\hookrightarrow\Mc(\Lg(G)^\nabla)$, 
\item a locally convex space $\Hc_{\Xi,\infty}$ 
such that there exists an injective continuous inclusion map 
$\Hc_{\Xi,\infty}\hookrightarrow\Hc$.  
\end{itemize}
Also let $\eta_{\Lg(G)}\colon\Lg(G)\to(\Lg(G)^\nabla)^\nabla$ be the uniformly continuous map defined in Remark~\ref{gen2}. 
\end{setting}

\begin{definition}\label{opcalc0.5}
\normalfont
We say that $\Gamma$ and $\theta$ are \emph{compatible} if 
the linear mapping  
$$\Fc_\Xi\colon\Gamma\to\UCb(\Xi),\quad 
\mu\mapsto\widehat{\mu}\circ\eta_{\Lg(G)}\circ\theta $$
is well defined and injective. 

If this is the case, then we denote $\Qc_\Xi:=\Fc_\Xi(\Gamma)\hookrightarrow\UCb(\Xi)$ 
and endow it with the topology which makes the \emph{Fourier transform} 
$$\Fc_\Xi\colon\Gamma\to\Qc_\Xi $$
into a linear toplogical isomorphism. 
We then also have the linear toplogical isomorphism 
$(\Fc_\Xi')^{-1}\colon\Gamma'\to\Qc_\Xi'$.
\end{definition}

\begin{lemma}\label{opcalc0.6}
If\ \ $\Gamma$ and $\theta$ are compatible, then the following conditions are equivalent: 
\begin{enumerate}
\item\label{opcalc0.6_item1} 
We have the well-defined continuous sesquilinear mapping 
$$\Ac^{\pi,\theta}\colon\Hc_{\Xi,\infty}\times\Hc_{\Xi,\infty}\to\Qc_\Xi,\quad 
(\phi,\psi)\mapsto\Ac^{\pi,\theta}_\psi\phi
:=(\pi(\exp_G(\theta(\cdot)))\phi\mid\psi).$$
\item\label{opcalc0.6_item2} 
There exists a unique continuous sesquilinear mapping
\begin{equation*}%\label{opcalc3_eq1}
\Wig\colon\Hc_{\Xi,\infty}\times\Hc_{\Xi,\infty}\to\Gamma,
\end{equation*}
such that for all $\phi,\psi\in\Hc_{\Xi,\infty}$ we have  
$\Fc_\Xi(\Wig(\phi,\psi))=(\pi(\exp_G(\theta(\cdot)))\phi\mid\psi)$. 
\end{enumerate}
\end{lemma}

\begin{proof}
If condition \eqref{opcalc0.6_item2} is satisfied, then 
$\Ac^{\pi,\theta}=\Fc_\Xi\circ\Wig$, hence condition~\eqref{opcalc0.6_item1} 
follows since $\Fc_\Xi\colon\Gamma\to\Qc_\Xi $ is a linear topological isomorphism. 
For the same reason, it also follows that if condition~\eqref{opcalc0.6_item1} 
holds true, then \eqref{opcalc0.6_item2} is satisfied. 
\end{proof}

\begin{definition}\label{opcalc1}
\normalfont
Assume that $\Gamma$ and $\theta$ are compatible and the equivalent conditions in Lemma~\ref{opcalc0.6} are satisfied. 
Then the sesquilinear map $\Wig$ is called the \emph{Wigner transform}.
The \emph{operator calculus for $\pi$ along~$\theta$} %\index{operator calculus} 
is the linear map 
$$\Op^\theta\colon\Gamma'\to\Lc(\Hc_{\Xi,\infty},\overline{\Hc}_{\Xi,\infty}')$$ 
defined by 
\begin{equation}\label{opcalc_eq1}
(\Op^\theta(a)\phi\mid\psi)
=\langle\underbrace{(\Fc_\Xi')^{-1}(a)}_{\hskip20pt\in\Qc_\Xi'},
\underbrace{(\pi(\exp_G(\theta(\cdot)))\phi\mid\psi)}_{\hskip20pt\in\Qc_\Xi}\rangle
\end{equation}
for $a\in\Gamma'$ and $\phi,\psi\in\Hc_{\Xi,\infty}$, 
where $\overline{\Hc}_{\Xi,\infty}'$ denotes the space of continuous antilinear functionals on ${\Hc}_{\Xi,\infty}$. 
\end{definition}

\begin{remark}\label{opcalc3}
\normalfont
In the setting of Definition~\ref{opcalc1} we have 
for all $a\in\Gamma'$ and $\phi,\psi\in\Hc_{\Xi,\infty}$  
$$(\Op^\theta(a)\phi\mid\psi)
=\langle(\Fc_\Xi')^{-1}(a),\Ac^{\pi,\theta}_\psi\phi\rangle
=\langle a,\Fc_\Xi^{-1}(\Ac^{\pi,\theta}_\psi\phi)\rangle
=\langle a,\Wig(\phi,\psi)\rangle,$$
where the later equality follows by 
Lemma~\ref{opcalc0.6}\eqref{opcalc0.6_item2}.
\end{remark}

\begin{definition}\label{opcalc4}
\normalfont
Assume the setting of Definition~\ref{opcalc1}. 
We say that the representation $\pi$ satisfies 
the \emph{orthogonality relations}\index{orthogonality relations} 
along the mapping $\theta\colon\Xi\to\Lg(G)$ 
if the following conditions are satisfied: 
\begin{enumerate}
\item 
The linear subspace $\Hc_{\infty,\Xi}$ is dense in $\Hc$. 
\item 
There exists a continuous, positive definite, sesquilinear, inner product on $\Gamma$ such that 
if we denote by $\Gamma_2$ the corresponding Hilbert space obtained by completion, 
then the sesquilinear mapping $\Wig\colon\Hc_{\Xi,\infty}\times\Hc_{\Xi,\infty}\to\Gamma$ 
extends to a unitary operator 
\begin{equation}\label{opcalc4_eq1}
\Wig\colon\Hc\botimes\overline{\Hc}\to\Gamma_2
\end{equation}  
which is still called the \emph{Wigner transform}. 
\end{enumerate}
\end{definition}

\begin{remark}\label{opcalc5}
\normalfont
In Definition~\ref{opcalc4}, since the inner product on $\Gamma$ is continuous, 
it gives rise to a continuous injective map $\overline{\Gamma}_2\hookrightarrow\Gamma'$. 
By using the fact that the operator~\eqref{opcalc4_eq1} is an isometry, 
we easily get 
$$(\forall\phi,\psi\in\Hc)\quad 
\Op^\theta(\Wig(\phi,\psi))=(\cdot\mid\psi)\phi. $$
\end{remark}

\section{Some computations involving Gaussian measures}\label{Sect4}

This section records some auxiliary facts that will be needed in the proof of Theorem~\ref{orth_th}. 

\begin{notation}\label{ind0}
\normalfont 
We shall use the following notation: 
\begin{enumerate}
\item For $x=(x_1,\dots,x_m)\in\RR^m$ we set  
$x'=(x_1,\dots,x_{m-1})\in\RR^{m-1}$ 
if $m\ge2$, 
hence 
$$x=(x',x_m)\in\RR^{m-1}\times\RR.$$ 
\item For every $t>0$ we denote 
$$(\forall x\in\RR)\quad 
\gamma_{t}(x)=\Bigl(\frac{1}{2\pi t}\Bigr)^{1/2}
\,\ee^{-\frac{x^2}{2t}},$$
so that $\gamma_{t}(x)\de x$ is the centered Gaussian probability measure on $\RR$ with \emph{variance} $t$ 
(and mean~$0$). 
\end{enumerate}
\end{notation}

\begin{remark}\label{ind1}
\normalfont
We recall that 
\begin{equation}\label{ind1_eq1}
\int\limits_{\RR}\gamma_{t}(x)\ee^{\ie vx}\de x
=\Bigl(\frac{2\pi}{t}\Bigr)^{1/2}\gamma_{1/t}(v)
=\ee^{-\frac{tv^2}{2}}
\end{equation}
for all $v\in\RR$ and $t>0$. 
\end{remark}

\begin{remark}\label{ind2}
\normalfont
For any integer $m\ge1$ recall the unitary operator that gives the integral kernels of operators obtained 
by classical Weyl calculus with $L^2$-symbols
$$T_m\colon L^2(\RR^m\times(\RR^m)')\to L^2(\RR^m\times\RR^m)$$
defined by 
$$(T_m(a))(x,y)=
\frac{1}{(2\pi)^{m/2}}\int\limits_{\RR^m} 
a(\frac{x+y}{2},\xi)\ee^{\ie\langle x-y,\xi\rangle}\de\xi,$$
with its inverse 
$$T_m^{-1}\colon L^2(\RR^m\times\RR^m)\to L^2(\RR^m\times(\RR^m)') $$
given by 
\begin{equation}\label{ind2_eq1}
(T_m^{-1}(K))(x,\xi)=\frac{1}{(2\pi)^{m/2}}\int\limits_{\RR^m} 
K(x+\frac{v}{2},x-\frac{v}{2})\ee^{-\ie\langle v,\xi\rangle}\de v, 
\end{equation}
for every $m\ge 1$. 

Then for arbitrary $m_1,m_2\ge 1$ there exist the natural unitary operators 
$$V\colon L^2(\RR^{m_1}\times\RR^{m_1})\botimes 
L^2(\RR^{m_2}\times\RR^{m_2}) \to 
L^2(\RR^{m_1+m_2}\times\RR^{m_1+m_2}) $$
given by 
$$(V(f_1\otimes f_2))((x_1,x_2),(y_1,y_2))=f_1(x_1,y_1)f_2(x_2,y_2)$$ 
for $x_j,y_j\in\RR^{m_j}$, $f_j\in L^2(\RR^{m_j}\times\RR^{m_j})$, $j=1,2$, 
and 
$$W\colon L^2(\RR^{m_1}\times(\RR^{m_1})')\botimes 
L^2(\RR^{m_2}\times(\RR^{m_2})') \to 
L^2(\RR^{m_1+m_2}\times(\RR^{m_1+m_2})')  $$
given by 
$$(W(g_1\otimes g_2))((x_1,x_2),(\xi_1,\xi_2))=g_1(x_1,\xi_1)g_2(x_2,\xi_2)$$ 
for $x_j\in\RR^{m_j}$, $\xi_j\in L^2((\RR^{m_j})')$, 
$f_j\in L^2(\RR^{m_j}\times(\RR^{m_j})')$, $j=1,2$,
and the diagram 
$$\xymatrix{ L^2(\RR^{m_1}\times(\RR^{m_1})')\botimes 
L^2(\RR^{m_2}\times(\RR^{m_2})') 
\ar[d]_{T_{m_1}\otimes T_{m_2}} \ar[r]^{\hskip25pt W}
 & L^2(\RR^{m_1+m_2}\times(\RR^{m_1+m_2})') \ar[d]^{T_{m_1+m_2}}
\\
L^2(\RR^{m_1}\times\RR^{m_1})\botimes 
L^2(\RR^{m_2}\times\RR^{m_2}) \ar[r]^{\hskip25pt V}
 & L^2(\RR^{m_1+m_2}\times\RR^{m_1+m_2}) 
}$$
is commutative. 
\end{remark}

The following simple lemma plays a key role in the present paper since it allows us 
to determine integral kernels of operators on Gaussian $L^2$-spaces in terms of integral kernels on Lebesgue $L^2$-spaces. 

\begin{lemma}\label{ind2.5}
The unitary operator $T_1\colon L^2(\RR^2)\to L^2(\RR^2)$ 
defined by 
$$(\forall a\in L^2(\RR^2))\quad (T_1(a))(x,y)=
\frac{1}{(2\pi)^{1/2}}\int\limits_{\RR} 
a(\frac{x+y}{2},\xi)\ee^{\ie(x-y)\xi}\de\xi$$ 
has the property 
$$(\forall t>0)\quad T_1^{-1}((\gamma_{t}\otimes\gamma_{t})^{1/2})
=(\gamma_{t/2}\otimes\gamma_{1/8t})^{1/2}.$$
\end{lemma}

\begin{proof}
Let us denote $K=(\gamma_{t}\otimes\gamma_{t})^{1/2}$. 
Then we have 
$$K(x,y)=\Bigl(\frac{1}{2\pi t}\Bigr)^{1/2}
\,\ee^{-\frac{x^2+y^2}{4t}} $$
hence, using \eqref{ind2_eq1}, one obtains 
\allowdisplaybreaks
\begin{align}
(T_1^{-1}(K))(x,\xi)
%&=\frac{1}{(2\pi)^{1/2}}\int\limits_{\RR} K\bigl(x+\frac{v}{2},x-\frac{v}{2}\bigr) \ee^{-\ie v\xi}\de v \nonumber\\
&=\frac{1}{2\pi t^{1/2}}\int\limits_{\RR} \ee^{-\frac{1}{4t}\bigl((x+\frac{v}{2})^2+(x-\frac{v}{2})^2\bigr)} 
\ee^{-\ie v\xi}\de v \nonumber\\
%&=\frac{1}{2\pi t^{1/2}}\int\limits_{\RR} \ee^{-\frac{1}{2t}\bigl(x^2+\frac{v^2}{4}\bigr)} 
%\ee^{-\ie v\xi}\de v \nonumber\\
&=\frac{1}{2\pi t^{1/2}}\, \ee^{-\frac{x^2}{2t}} 
\int\limits_{\RR} \ee^{-\frac{v^2}{8t}} 
\ee^{-\ie v\xi}\de v \nonumber\\
&=\frac{1}{2\pi t^{1/2}}\, \ee^{-\frac{x^2}{2t}} (8\pi t)^{1/2}
\int\limits_{\RR} \gamma_{4t}(v)
\ee^{\ie v\xi}\de v \nonumber\\
&=\Bigl(\frac{2}{\pi}\Bigr)^{1/2} \ee^{-\frac{x^2}{2t}} \ee^{-2t \xi^2} \nonumber\\
&=\Bigl(\Bigl(\frac{1}{\pi t}\Bigr)^{1/2} \ee^{-\frac{x^2}{t}} 
\Bigl(\frac{4t}{\pi}\Bigr)^{1/2}\ee^{-4t \xi^2}\Bigr)^{1/2} \nonumber\\
&=(\gamma_{t/2}(x)\gamma_{1/8t}(\xi))^{1/2} \nonumber
\end{align}
where the third from last equality follows by \eqref{ind1_eq1}, 
and this completes the proof. 
\end{proof}

\begin{proposition}\label{ind3}
Let $t_1\ge t_2\ge\cdots>0$ and  
for $m=1,2,\dots$ define
$$\widetilde{\gamma}_m\colon\RR^m\to\RR,\quad 
\widetilde{\gamma}_m(x_1,\dots,x_m):=\gamma_{t_1}(x_1)\cdots\gamma_{t_m}(x_m) $$
and $\de\widetilde{\gamma}_m:=\widetilde{\gamma}_m(x)\de x$. 
Then for $m\ge 2$ the diagram 
$$\xymatrix{
L^2(\RR^{m-1}\times(\RR^{m-1})') \ar[d]_{S_{m-1}} \ar[r]^{\hskip15pt\beta}
 & L^2(\RR^{m}\times(\RR^{m})') \ar[d]^{S_{m}}\\
L^2(\RR^{m-1}\times(\RR^{m-1})') \ar[d]_{T_{m-1}} \ar[r]^{\hskip15pt\alpha}
 & L^2(\RR^{m}\times(\RR^{m})') \ar[d]^{T_{m}}\\
L^2(\RR^{m-1}\times\RR^{m-1}) \ar[d]_{U_{m-1}} \ar[r]^{\hskip15pt\eta}
 & L^2(\RR^{m}\times\RR^{m}) \ar[d]^{U_{m}} \\
L^2(\RR^{m-1}\times\RR^{m-1},\de\widetilde{\gamma}_{m-1}\times\de\widetilde{\gamma}_{m-1}) \ar[r]^{\hskip15pt\iota}
 & L^2(\RR^{m}\times\RR^{m},\de\widetilde{\gamma}_{m}\times\de\widetilde{\gamma}_{m}) 
} $$
is commutative, where the vertical arrows are unitary operators 
defined by 
$$\begin{aligned}
(S_{\ell}(b))(x_1,\dots,x_{\ell},\xi_1,\dots,\xi_{\ell})
&= \frac{1}{(t_1\cdots t_{\ell})^2}\cdot b(\frac{x_1}{t_1},\dots,\frac{x_{\ell}}{t_{\ell}},
\frac{\xi_1}{t_1},\dots,\frac{\xi_{\ell}}{t_{\ell}})\\
(T_{\ell}(a))(x,y)
&=
\frac{1}{(2\pi)^{l/2}}\int\limits_{\RR^l} 
a(\frac{x+y}{2},\xi)\ee^{\ie\langle x-y,\xi\rangle}\de\xi, \\
(U_{\ell}(K))(v,w)
&=K(v,w)\widetilde{\gamma}_{\ell}(v)^{-1/2}\widetilde{\gamma}_{\ell}(w)^{-1/2}, 
\end{aligned}$$
for $\ell\in\{m-1,m\}$ and the horizontal arrows are isometries defined by 
\allowdisplaybreaks
\begin{align}
(\beta(b))(x,\xi)
&=b(x',\xi')\gamma_{1/2}(x_m)^{1/2}\gamma_{1/8t_m^2}(\xi_m)^{1/2} \nonumber \\
(\alpha(a))(x,\xi)
&=a(x',\xi')\gamma_{t_m/2}(x_m)^{1/2}\gamma_{1/8t_m}(\xi_m)^{1/2} \nonumber\\
(\eta(K))(x,y) 
&=K(x',y')\gamma_{t_m}(x_m)^{1/2}\gamma_{t_m}(y_m)^{1/2} \nonumber\\
(\iota(Q))(v,w) 
&= Q(v',w'). \nonumber
\end{align} 
\end{proposition}

\begin{proof}
It is clear that the horizontal arrows in the diagram are isometries, 
while the vertical arrows are unitary operators. 

Moreover, for $K\in L^2(\RR^{m-1}\times\RR^{m-1})$ we have 
\allowdisplaybreaks
\begin{align}
(U_m(\eta(K)))(v,w)
&=(\eta(K))(v,w)\widetilde{\gamma}_m(v)^{-1/2}\widetilde{\gamma}_m(w)^{-1/2} \nonumber\\
&=K(v',w')\gamma_{t_m}(v_m)^{1/2}\gamma_{t_m}(w_m)^{1/2}\widetilde{\gamma}_m(v)^{-1/2}\widetilde{\gamma}_m(w)^{-1/2} \nonumber\\
&=K(v',w')\widetilde{\gamma}_{m-1}(v')^{-1/2}\widetilde{\gamma}_{m-1}(w')^{-1/2} \nonumber\\
&=(U_{m-1}(K))(v',w') \nonumber\\
&=(\iota(U_{m-1}(K)))(v,w), \nonumber
\end{align} 
hence the lower part of the diagram in the statement is commutative. 
As regards the middle part of the diagram, we have 
\allowdisplaybreaks
\begin{align}
(\eta(T_{m-1}(a)))(x,y) 
&=(T_{m-1}(a))(x',y')\gamma_{t_m}(x_m)^{1/2}\gamma_{t_m}(y_m)^{1/2} 
\nonumber\\
&=(T_{m-1}(a))(x',y')\cdot 
(T_1((\gamma_{t_m/2}\otimes\gamma_{1/8t_m})^{1/2}) )(x_m,y_m) 
\nonumber\\
&=(T_m(a\otimes(\gamma_{t_m/2}\otimes\gamma_{1/8t_m})^{1/2}))
(x',x_m,y',y_m) \nonumber\\
&=(T_m(\alpha(a)))(x,y) \nonumber
\end{align} 
where the second equality follows by Lemma~\ref{ind2.5}, 
while the third equality is a consequence of Remark~\ref{ind2}. 

Finally, by using the fact that  
$$(\forall t,a>0)(\forall x\in\RR) \quad 
a^{1/2}\gamma_{t}(ax)=\gamma_{t/a}(x),$$
it follows by a straightforward computation that the upper part of the diagram in the statement is commutative. 
\end{proof}

\begin{remark}
\normalfont
In order to point out the role of the unitary operators 
$$U_\ell\colon L^2(\RR^\ell\times\RR^\ell)\to L^2(\RR^\ell\times\RR^\ell,\de\widetilde{\gamma}_\ell\times\de\widetilde{\gamma}_\ell)$$ 
in Proposition~\ref{ind3}, we note the following fact. 
The multiplication operator
$$V_\ell\colon L^2(\RR^\ell)\to L^2(\RR^\ell,\de\widetilde{\gamma}_\ell),\quad 
V_\ell f=f(\widetilde{\gamma}_\ell)^{-1/2}$$
is unitary. 
Furthermore, for any $K\in L^2(\RR^\ell\times\RR^\ell)$, 
if we denote the corresponding integral operator by 
$$A_K\colon L^2(\RR^\ell)\to L^2(\RR^\ell), \quad (A_K f)(x)=\int\limits_{\RR^\ell}K(x,y)f(y)\de y,$$
then $V_\ell A_K V_\ell^{-1}\colon L^2(\RR^\ell,\de\widetilde{\gamma}_\ell)\to L^2(\RR^\ell,\de\widetilde{\gamma}_\ell)$ 
is the integral operator whose integral kernel is $U_\ell(K)\in L^2(\RR^\ell\times\RR^\ell,\de\widetilde{\gamma}_\ell\times\de\widetilde{\gamma}_\ell)$. 
\end{remark}

\section{Infinite-dimensional Heisenberg groups and Wigner transforms}\label{Sect5}

In this section we study representations of a special type of infinite-dimensional Heisenberg groups, 
as introduced in the following definition. 
A more general framework can be found for instance in \cite{Nee00} and \cite{DG08}. 

For the sake of completeness we recall here a few facts from \cite[Sect. 3]{BB10c}. 

\begin{definition}\label{heisenberg}
\normalfont
If $\Vc$ is a real Hilbert space, $A\in\Bc(\Vc)$ with 
$(Ax\mid y)=(x\mid Ay)$ for all $x,y\in\Vc$, 
and moreover $\Ker A=\{0\}$, then  
the \emph{Heisenberg algebra} associated with the pair $(\Vc,A)$ is  
the real Hilbert space $\hg(\Vc,A)=\Vc\dotplus\Vc\dotplus\RR$ endowed with the Lie bracket  
defined by 
$[(x_1,y_1,t_1),(x_2,y_2,t_2)]=(0,0,(Ax_1\mid y_2)-(Ax_2\mid y_1))$.   
The corresponding \emph{Heisenberg group} $\HH(\Vc,A)=(\hg(\Vc,A),\ast)$ 
is the Lie group whose underlying manifold is $\hg(\Vc,A)$ and 
whose multiplication is defined by 
$$(x_1,y_1,t_1)\ast(x_2,y_2,t_2)=
(x_1+x_2,y_1+y_2,t_1+t_2+((Ax_1\mid y_2)-(Ax_2\mid y_1))/2) $$
for $(x_1,y_1,t_1),(x_2,y_2,t_2)\in\HH(\Vc,A)$. 
\end{definition}

\subsection*{Gaussian measures and Schr\"odinger representations}
Let $\Vc_{-}$ be a real Hil\-bert space with the scalar product denoted by $(\cdot\mid\cdot)_{-}$. 
For every vector $a\in\Vc_{-}$ and 
every symmetric, nonnegative, injective, \emph{trace-class} operator $K$ on $\Vc_{-}$ 
there exists a unique probability Borel measure~$\gamma$ on $\Vc_{-}$
such that 
$$(\forall x\in\Vc_{-})\quad 
\int\limits_{\Vc_{-}}\ee^{\ie(x\mid y)_{-}}\de\gamma(y)
=\ee^{\ie(a\mid x)_{-}-\frac{1}{2}(Kx\mid x)_{-}} $$
(see for instance \cite[Th. I.2.3]{Kuo75}). 
We also have  
$$a=\int\limits_{\Vc_{-}}y\ \de\gamma(y)
\quad\text{ and }\quad 
Kx=\int\limits_{\Vc_{-}}(x\mid y)_{-}\cdot(y-a)\de\gamma(y) 
\text{ for all }x\in\Vc_{-}, $$
where the integrals are weakly convergent, and $\gamma$ is called the  
\emph{Gaussian measure with the mean $a$ and the variance $K$}.  

Let us assume that the Gaussian measure $\gamma$ is centered, that is, $a=0$. 
Denote $\Vc_{+}:=\Ran K$ and $\Vc_0:=\Ran K^{1/2}$ endowed 
with the scalar products $(Kx\mid Ky)_{+}:=(x\mid y)_{-}$ and 
$(K^{1/2}x\mid K^{1/2}y)_0:=(x\mid y)_{-}$, respectively, for all $x,y\in\Vc_{-}$, 
which turn the linear bijections 
$K\colon\Vc_{-}\to\Vc_{+}$ and $K^{1/2}\colon\Vc_{-}\to\Vc_0$ 
into isometries. 
We thus get the real Hilbert spaces
$$\Vc_{+}\hookrightarrow\Vc_0\hookrightarrow\Vc_{-} $$
where the inclusion maps are Hilbert-Schmidt operators, 
since so is $K^{1/2}\in\Bc(\Vc_{-})$. 
Also, the scalar product of $\Vc_0$ extends to a duality pairing $(\cdot\mid\cdot)_0\colon\Vc_{-}\times\Vc_{+}\to\RR$. 

We also recall that for every $x\in\Vc_{+}$ the translated measure $\de\gamma(-x+\cdot)$ 
is absolutely continuous with respect to $\de\gamma(\cdot)$ 
and we have the Cameron-Martin formula
$$\de\gamma(-x+\cdot)=\rho_x(\cdot)\de\gamma(\cdot) 
\quad\text{ with }\rho_x(\cdot)=\ee^{(\cdot\mid x)_0-\frac{1}{2}(x\mid x)_0}.$$  
(This actually holds true for every $x\in\Vc_0$, 
by suitably defining the function $\rho_x(\cdot)$ no longer as a continuous function, 
but only almost everywhere; 
see for instance 
\cite[Lemma I.4.7 and Th. II.3.1]{Kuo75}.)

\begin{definition}\label{sch_def}
\normalfont 
Let $\Vc_{+}$ be a real Hilbert space with the scalar product denoted by $(x,y)\mapsto (x\mid y)_{+}$. 
Also let $A\colon\Vc_{+}\to\Vc_{+}$ be a nonnegative, symmetric, injective, trace-class operator.  
Define $\Vc_0$ and $\Vc_{-}$ as the completions of $\Vc_{+}$ with respect to the scalar products  
$$(x,y)\mapsto (x\mid y)_0:=(A^{1/2}x\mid A^{1/2}y)_{+}$$ 
and 
$$(x,y)\mapsto (x\mid y)_{-}:=(Ax\mid Ay)_{+},$$ 
respectively.  
Then the operator $A$ uniquely extends to a nonnegative, symmetric, injective, trace-class operator $K\in\Bc(\Vc_{-})$,   
hence by the above observations one obtains the centered Gaussian measure $\gamma$ on $\Vc_{-}$ with the variance $K$. 

One can also construct the Heisenberg group $\HH(\Vc_{+},A)$. 
The \emph{Schr\"odinger representation} 
$\pi\colon\HH(\Vc_{+},A)\to\Bc(L^2(\Vc_{-},\gamma))$ is defined by 
$$\pi(x,y,t)\phi=\rho_x(\cdot)^{1/2}
\ee^{\ie(t+(\cdot\mid y)_0+\frac{1}{2}(x\mid y)_0)}\phi(-x+\cdot) $$
whenever $(x,y,t)\in\HH(\Vc_{+},A)$ and $\phi\in L^2(\Vc_{-},\gamma)$. 
This is a continuous unitary irreducible representation 
of the Heisenberg group $\HH(\Vc_{+},A)$ 
by \cite[Th.~5.2.9 and 5.2.10]{Hob06} and 
Proposition~\ref{sch_irred} below.
\end{definition}

\begin{remark}\label{sch_neeb}
\normalfont
More general Schr\"odinger representations of infinite-dimensional Hei\-senberg groups are 
described in \cite[Prop.~II.4.6]{Nee00} by using cocycles and reproducing kernel Hilbert spaces. 
\end{remark}

The following result is known, but we recall here from \cite[Rem. 3.6]{BB10c} the method of proof 
since its constructions will be needed also for Theorem~\ref{orth_th}.

\begin{proposition}
\label{sch_irred}
The representation 
$\pi\colon\HH(\Vc_{+},A)\to\Bc(L^2(\Vc_{-},\gamma))$ 
from Definition~\ref{sch_def} is irreducible. 
\end{proposition} 

\begin{proof}
Let $t_1\ge t_2\ge\cdots$ ($>0$) be the eigenvalues of $A$ counted according to their multiplicities. 
Since $A$ is a self-adjoint trace-class operator and $\Ker A=\{0\}$, there exists an orthonormal basis $\{v_k\}_{k\ge1}$ in $\Vc_{+}$ such that $Av_k=t_k v_k$ for every $k\ge1$. 
For every integer $n\ge1$ let 
$$\Vc_{n,+}=\spa\{v_1,\dots,v_n\}.$$ 
It follows that $\dim\Vc_{n,+}<\infty$. 
We have 
$$\Vc_{1,+}\subseteq\Vc_{2,+}\subseteq\cdots\subseteq
\bigcup\limits_{n\ge1}\Vc_{n,+}\subseteq\Vc_{+}$$
and $\bigcup\limits_{n\ge1}\Vc_{n,+}$ is a dense subspace of $\Vc_{+}$. 
Let us denote by $A_n$ the restriction of $A$ to $\Vc_{n,+}$. 
Then $\HH(\Vc_{n,+},A_n)$ is a finite-dimensional Heisenberg group, 
hence it is well known that its Schr\"odinger representation 
$\pi_n\colon\HH(\Vc_{n,+},A_n)\to\Bc(L^2(\Vc_{n,-},\gamma_n))$ 
is irreducible, where $\gamma_n$ is the Gaussian measure on the finite-dimensional space $\Vc_{n,-}$ obtained out of 
the pair $(\Vc_{+,n},A_n)$ 
by the construction outlined at the very beginning of Definition~\ref{sch_def}. 
Note that 
$$\HH(\Vc_{1,+},A_1)\subseteq\HH(\Vc_{2,+},A_2)\subseteq\cdots\subseteq
\bigcup\limits_{n\ge1}\HH(\Vc_{n,+},A_n)=:\HH(\Vc_{\infty,+},A_\infty)
\subseteq \HH(\Vc_{+},A) $$
and $\HH(\Vc_{\infty,+},A_\infty)$ is a dense subgroup of $\HH(\Vc_{+},A)$, 
hence the Schr\"odinger representation $\pi\colon\HH(\Vc_{+},A)\to\Bc(L^2(\Vc_{-},\gamma))$ is irreducible if and only if so is its restriction $\pi\vert_{\HH(\Vc_{\infty,+},A_\infty)}$. 

On the other hand, if we denote by $\1_n$ the function identically equal to 1 on the orthogonal complement $\Vc_{n+1,+}\ominus\Vc_{n,+}$, then it is straightforward to check that the operator 
\begin{equation}\label{sch_irred_eq1}
L^2(\Vc_{n,-},\gamma_n)\to L^2(\Vc_{n+1,-},\gamma_{n+1}), 
\quad f\mapsto f\otimes \1_n
\end{equation}
is an isometry and intertwines the representations $\pi_n$ and $\pi_{n+1}$. 
We can thus make the sequence of representations $\{\pi_n\}_{n\ge1}$ into an inductive system of irreducible unitary representations and then their inductive limit $\pi\vert_{\HH(\Vc_{\infty,+},A_\infty)}
=\mathop{\rm ind}\limits_{n\to\infty}\pi_n$ is irreducible 
(see for instance \cite{KS77}). 
As noted above, this implies that the Schr\"odinger representation 
$\pi\colon\HH(\Vc_{+},A)\to\Bc(L^2(\Vc_{-},\gamma))$ 
of Definition~\ref{sch_def} is irreducible. 
\end{proof}

\begin{remark}\label{algdual}
\normalfont
Let $\Xc$ be a real locally convex Hausdorff space and denote by $\Xc'^{\alg}$ 
the space of all linear functionals on~$\Xc$ endowed with the locally convex topology of pointwise convergence on~$\Xc$. 
If $B$ is any algebraic basis in $\Xc$, then 
every $\xi=\{\xi_b\}_{b\in B}\in\RR^B$ corresponds to a linear functional 
$$\Psi_B(\xi)\colon\Xc\to\RR,\quad 
\sum\limits_{b\in B}x_b b\mapsto \sum\limits_{b\in B}\xi_b x_b b,$$
and we get the linear topological isomorphism 
$$\Psi_B\colon\RR^B\to \Xc'^{\alg},\quad 
\xi\mapsto \Psi_B(\xi)$$
where $\RR^B$ is endowed with its natural weak topology 
(that is, projective limit of finite-dimensional linear spaces; 
see \cite[Def.~A2.5]{HM07}). 

For later use we also note that if the topological dual $\Xc'$ 
is endowed with the weak$^*$-topology, then the following assertions hold: 
\begin{enumerate}
\item\label{dualalg_item1} 
The locally convex space $\Xc'^{\alg}$ is complete and 
the topological dual $\Xc'$ is a dense subspace of $\Xc'^{\alg}$. 
Therefore $\Xc'^{\alg}$ is isomorphic to the completion of $\Xc'$ 
as a topological vector space (hence as a uniform space). 
\item\label{dualalg_item2} 
The mapping 
$$\UCb(\Xc'^{\alg})\to\UCb(\Xc'),\quad f\mapsto f\vert_{\Xc'} $$
is an isometric $*$-isomorphism of unital $C^*$-algebras. 
\item\label{dualalg_item3}
If $B_0\subseteq B$ and we consider the natural injective linear map 
$$\iota_{B_0,B}\colon \RR^{B_0}\hookrightarrow\RR^B
\mathop{\longrightarrow}\limits^{\Psi_B}\Xc'^{\alg}$$
then we get a mapping 
$$\UCb(\Xc'^{\alg})\to\UCb(\RR^{B_0}),\quad f\mapsto f\circ\iota_{B_0,B} $$
which is a surjective $*$-homomorphism of unital $C^*$-algebras. 
\end{enumerate}
In fact, Assertion~\eqref{dualalg_item1} follows by a result of \cite{Gro50}; 
see also \cite{Dev72}, \cite[Sol.~30]{Hal82}.  
A version of this property for general uniform spaces was established 
in~\cite{Feo61}. 
Assertion~\eqref{dualalg_item2} follows by using the fact that the uniformly continuous functions extend uniquely to the completion; 
see for instance~\cite[Ch.~II, \S 3]{Bo71}. 
As for Assertion~\eqref{dualalg_item3}, 
we just have to note that if $p_{B,B_0}\colon\RR^B\to\RR^{B_0}$ denotes the natural projection, 
then $(p_{B,B_0}\circ(\Psi_B)^{-1})\circ\iota_{B_0,B}=\id_{\RR^{B_0}}$. 
\end{remark}

In the following statement we use notation from Definition~\ref{sch_def} 
and moreover we denote by $\Sg_2(\cdot)$ the Hilbert space of Hilbert-Schmidt operators on some complex Hilbert space. 

\begin{theorem}\label{orth_th}
Let $\Vc_{+}$ be a separable real Hilbert space and 
$A\colon\Vc_{+}\to\Vc_{+}$ be a nonnegative, symmetric, injective, trace-class operator with the eigenvalues 
$$t_0\ge t_1\ge\cdots>0,$$ 
counted according to their multiplicities. 
Let $\Vc_{\infty,+}$ denote the linear span of the eigenvectors of $A$ and define 
$$\theta\colon\Xi:=\Vc_{\infty,+}\times\Vc_{\infty,+}\to\hg(\Vc_{+},A), \quad 
(x,y)\mapsto(x,y,0). $$
Then there exist the locally convex spaces $\Gamma\hookrightarrow\Mc(\hg(\Vc_{+},A)')$ 
and $\Hc_{\Xi,\infty}\hookrightarrow\Hc:=L^2(\Vc_{-},\gamma)$ 
such that the following assertions hold: 
\begin{enumerate}
\item\label{orth_th_item1} 
$\Gamma$ and $\theta$ are compatible. 
\item\label{orth_th_item2} 
The Schr\"odinger representation 
$\pi\colon\HH(\Vc_{+},A)\to\Bc(L^2(\Vc_{-},\gamma))$ satisfies the orthogonality relations along $\theta$. 
\item\label{orth_th_item3} For the Hilbert space obtained as the completion of $\Gamma$ we have 
$$\Gamma_2\simeq 
L^2\Bigl(\RR^{\NN}\times\RR^{\NN},
\Bigl(\bigotimes_{j\in\NN}\de\gamma_{1/2}\Bigr)
\otimes\Bigl(\bigotimes_{j\in\NN}\de\gamma_{1/8t_j^2}\Bigr)\Bigr)
\hookrightarrow\Mc(\hg(\Vc_{+},A)') $$
and the operator $\Op^\theta\colon\Gamma_2\to\Sg_2(L^2(\Vc_{-},\gamma))$ defined in \eqref{opcalc_eq1} is unitary.
\end{enumerate} 
\end{theorem}

\begin{proof}
We shall use the notation of Remark~\ref{sch_irred}. 
For every $k\ge0$ we have $\Vert v_k\Vert_{-}=\Vert Av_k\Vert_{+}=t_k$ 
hence, if we denote $e_k=t_k^{-1}v_k$,   
then $\{e_k\}_{k\ge1}$ is an orthonormal basis in $\Vc_{-}$, 
and we use it to perform the identifications $\Vc_{-}\simeq\ell^2_{\RR}(\NN)$ 
and $\Vc_{m,-}\simeq\RR^m\hookrightarrow\ell^2_{\RR}(\NN)$ for $m=1,2,\dots$. 

For $m\ge 1$ let us denote by 
$$\Hc_{m,\infty}:=\{\varphi\in L^2(\Vc_{m,-},\widetilde{\gamma}_m)\mid \pi_m(\cdot)\varphi\in\Ci(\HH(\Vc_{m,+},A_m),L^2(\Vc_{m,-},\widetilde{\gamma}_m))\}$$ 
the space of smooth vectors for the representation 
$$\pi_m\colon\HH(\Vc_{m,+},A_m)\to\Bc(L^2(\Vc_{m,-},\widetilde{\gamma}_m))
=\Bc(L^2(\RR^m,\de \widetilde{\gamma}_m))$$
where $\de\widetilde{\gamma}_m$ is the Gaussian measure on $\RR^m$ 
as in Proposition~\ref{ind3}. 
(See for instance \cite{Nee10} and the references therein for differentiability of vectors in Lie group representation theory.)
It is clear that the operators of the type \eqref{sch_irred_eq1} take 
$\Hc_{m,\infty}$ into $\Hc_{m+1,\infty}$. 
We thus get an inductive limit of Fr\'echet spaces continuously and densely embedded into 
an inductive limit of Hilbert spaces, 
$$\Hc_{\Xi,\infty}:=\ind\limits_{m\to\infty}\Hc_{m,\infty}
\hookrightarrow 
\ind\limits_{m\to\infty}L^2(\Vc_{m,-},\widetilde{\gamma}_m)
=L^2(\Vc_{-},\gamma).$$
It is well known that for every $m\ge1$ the representation $\pi_m$ satisfies the orthogonality relations along the mapping 
$$\theta\vert_{\Hc_{m,\infty}\times\Hc_{m,\infty}}\colon 
\Hc_{m,\infty}\times\Hc_{m,\infty}\to\hg(\Vc_{m,+},A_m)$$
and the corresponding unitary operator (see \eqref{opcalc4_eq1}),  
denoted by 
$$\Wig_m\colon L^2(\Vc_{m,-},\widetilde{\gamma}_m)\botimes 
\overline{L^2(\Vc_{m,-},\widetilde{\gamma}_m)}
\to L^2(\RR^m\times\RR^m),$$
is given by 
$\Wig_m=(U_m\circ T_m\circ S_m)^{-1}$,  
where we use the notation of Proposition~\ref{ind3}. 
It follows by that proposition that for the Gaussian Radon measures 
$\mu_1:=\bigotimes\limits_{j\in\NN}\de\gamma_{1/2}$ and 
$\mu_2:=\bigotimes\limits_{j\in\NN}\de\gamma_{1/8t_j^2}$ on $\RR^{\NN}$ we obtain a unitary operator 
$$\Wig=\ind\limits_{m\to\infty}\Wig_m\colon  
L^2(\Vc_{-},\gamma)\botimes 
\overline{L^2(\Vc_{-},\gamma)}
\to L^2(\RR^{\NN}\times\RR^{\NN},\de \mu_1\times\de\mu_2).$$
By using Remark~\ref{algdual} we now obtain natural continuous injective linear maps 
\allowdisplaybreaks
\begin{align}
L^2(\RR^{\NN}\times\RR^{\NN},\de \mu_1\times\de\mu_2)
& \hookrightarrow 
L^1(\RR^{\NN}\times\RR^{\NN},\de \mu_1\times\de\mu_2) \nonumber\\
&\hookrightarrow 
\Mc(\RR^{\NN}\times\RR^{\NN}) \nonumber\\
&\hookrightarrow \Mc((\Vc_{+}\times\Vc_{+})'^{\alg}) \nonumber\\
&\simeq \Mc((\Vc_{+}\times\Vc_{+})') \nonumber\\
&\hookrightarrow \Mc(\hg(\Vc_{+},A)'). \nonumber
\end{align}
Finally, if we define 
$$\Gamma:=\Wig(\Hc_{\Xi,\infty}\otimes\overline{\Hc_{\Xi,\infty}})$$
endowed with the topology which makes the linear isomorphism 
$$\Wig\colon \Hc_{\Xi,\infty}\otimes\overline{\Hc_{\Xi,\infty}}\to\Gamma$$ 
into a topological isomorphism, 
then it is easily seen that $\Gamma$ and the mapping $\theta$ are compatible, 
and this concludes the proof. 
\end{proof}

\begin{remark}
\normalfont
The infinite-dimensional pseudo-differential calculus of \cite{AD96} and \cite{AD98} 
can be recovered as a  
special case of the operator calculus from our Definition~\ref{opcalc1} for 
the Schr\"odinger representations introduced in Definition~\ref{sch_def} above. 
Compare for instance \cite[Prop.~3.7]{AD98}. 
\end{remark}

%\subsection*{Acknowledgment} 

\end{document}